\documentclass[a4paper,english,reqno]{amsart}
\usepackage[utf8]{inputenc}
\usepackage{graphicx}
\usepackage{caption}

\usepackage{mathrsfs}
\usepackage{amsmath}
\usepackage{amsfonts}
\usepackage{amsthm}
\usepackage{amssymb}
\usepackage{amscd}
\usepackage{amstext}
\usepackage{mathtools}
\usepackage[all]{xy}
\xyoption{2cell}
\UseTwocells
\usepackage{stackengine}
\usepackage{calrsfs}
\usepackage{framed}
\usepackage[hidelinks]{hyperref}
\usepackage[noabbrev,capitalise]{cleveref}
\usepackage{url}
\usepackage{float}
\usepackage{tikz}
\usepackage{tikz-cd}
\usetikzlibrary{positioning,intersections,cd,matrix, arrows, backgrounds}
\usetikzlibrary{shapes,shapes.geometric,shapes.misc}
\usetikzlibrary{decorations}
\usetikzlibrary{decorations.pathreplacing}

\usepackage[OT2,T1]{fontenc}

\newcommand{\deff}{\textit}
\renewcommand\k{\Bbbk}
\newcommand\iso{\cong}
\newcommand{\Id}{\mathrm{id}}
\newcommand\op{^{\mathrm{op}}}
\newcommand\io{\iota}
\newcommand\ze{\zeta}

\newcommand\ZZ{\mathbb{ZZ}}
\newcommand\BL{\operatorname{BL}}
\newcommand\N{\mathbb{N}}
\newcommand\Z{\mathbb{Z}}
\newcommand\bbR{\mathbb{R}}
\newcommand\cU{{\mathbb U}}
\newcommand\R{\mathbb{R}}
\newcommand\bS{\mathbb{S}}
\newcommand\RR{\bbR\op\times \bbR}

\newcommand\ds{\oplus}

\newcommand\DS{\bigoplus\limits}
\newcommand{\cat}[1]{\ensuremath{\mathsf{#1}}}

\renewcommand{\Vec}{\cat{Vec}}

\theoremstyle{plain}
\newtheorem{theorem}{Theorem}[section]

\newtheorem{lemma}[theorem]{Lemma}

\theoremstyle{definition} 
\newtheorem{definition}[theorem]{Definition}
\newtheorem{example}[theorem]{Example}
\newtheorem{remark}[theorem]{Remark}

\providecommand\given{}
\newcommand\SetSymbol[1][]{%
  \nonscript\:#1\vert{}
  \allowbreak{}
  \nonscript\:
  \mathopen{}}
\DeclarePairedDelimiterX\Set[1]\{\}{%
  \renewcommand\given{\SetSymbol[]}
  #1
}

\newcommand\restr[2]{{
  \left.\kern-\nulldelimiterspace 
  #1 
  \vphantom{\big|} 
  \right|_{#2} 
  }}
\newcommand{\slice}[1]{{#1}'} 

\renewcommand{\subparagraph}[1]{%
  \vspace{1em}%
  \noindent\textbf{#1}\par%
  \vspace{0.5em}%
}

\begin{document}
\title{Bipath persistence as zigzag persistence}
\author{{\'A}ngel Javier Alonso and Enhao Liu}

\begin{abstract}
Persistence modules that decompose into interval modules are important in topological data analysis because we can interpret such intervals as the lifetime of topological features in the data. We can classify the settings in which persistence modules always decompose into intervals, by a recent result of Aoki, Escolar and Tada: these are standard single-parameter persistence, zigzag persistence, and bipath persistence. No other setting offers such guarantees.
  
We show that a bipath persistence module can be decomposed via a closely related infinite zigzag persistence module, understood as a covering. This allows us to translate techniques of zigzag persistence, like recent advancements in its efficient computation by Dey and Hou, to bipath persistence. In addition, and again by the relation with the infinite zigzag, we can define an interleaving and bottleneck distance on bipath persistence. In turn, the algebraic stability of zigzag persistence implies the algebraic stability of bipath persistence.
\end{abstract}

\makeatletter
\@namedef{subjclassname@2020}{\textup{2020} Mathematics Subject Classification}
\makeatother
\thanks{This research was funded in
whole, or in part, by the Austrian Science Fund (FWF) 10.55776/P33765. Enhao Liu is supported by JST SPRING, Grant Number JPMJSP2110.
}

\address{Institute of Geometry, Graz University of Technology, Graz, Austria}
\email{alonsohernandez@tugraz.at}

\address{
Department of Mathematics, Kyoto University,
Kitashirakawa Oiwake-cho, Sakyo-ku, Kyoto 606-8502, 
Japan.
}
\email{liu.enhao.b93@kyoto-u.jp}

\maketitle


\section{Introduction}\label{sec:introduction}

Intervals have an outsized role in persistent homology: we interpret them as the lifetime of topological features in the data. In standard single-parameter persistence, where we consider persistence modules over a totally ordered set, the modules always decompose into interval modules. The intervals in the decomposition are the \deff{barcode} of the module. The same happens in zigzag persistence, where we consider persistence modules over a zigzag partially ordered set (poset); all modules decompose into interval modules~\cite{botnanIntervalDecompositionInfinite2017}. Even in settings where modules do not always decompose into interval modules, like multiparameter persistence---where the underlying poset is a grid---, intervals are useful in applications. For example, consider the \deff{fibered barcode}~\cite{lesnickInteractiveVisualization2D2015}, which consists of single-parameter slices of a multiparameter module, and whose intervals, apart from being useful on their own, also give the \deff{multiparameter landscape}~\cite{vipondMultiparameterPersistenceLandscapes2020}. In this regard, it is unsurprising that much work on multiparameter persistence aims at understanding when multiparameter modules are interval-decomposable~\cite{caiElderRuleStaircodesAugmentedMetric2021,deyComputingGeneralizedRank2024,alonsoDecompositionZeroDimensionalPersistence2024,alonsoProbabilisticAnalysisMultiparameter2024a,asashibaIntervalMultiplicitiesPersistence2024,bauerGenericMultiParameterPersistence2023,asashibaIntervalDecomposability2D2022}, or at approximations by using the intervals~\cite{kimGeneralizedPersistenceDiagrams2021,lesnickInteractiveVisualization2D2015,asashibaApproximationIntervaldecomposablesInterval2023,asashiba2024interval,botnanSignedBarcodesMultiParameter2022,botnanSignedBarcodesMultiparameter2024}.

We can classify the settings where we always have interval decomposability by a recent result of Aoki, Escolar and Tada~\cite{aokiSummandinjectivityIntervalCovers2023}: persistence modules over a finite poset always decompose into interval modules if and only if the poset is
a totally ordered set (single-parameter persistence), a zigzag poset or a bipath
poset. A \deff{bipath poset} is the disjoint union of two finite totally ordered
sets, with a common bottom and top element added; see~\cref{fig:posets}.
\begin{figure}[h]
  \centering
  \includegraphics{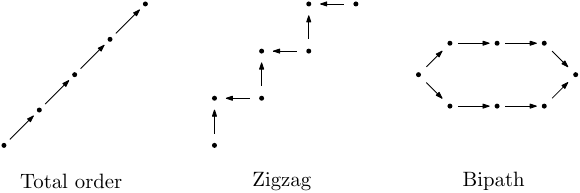}
  \caption{Posets that decompose as an interval}\label{fig:posets}
\end{figure}

One can obtain a bipath module---a persistence module over a bipath poset---by an order-preserving map between a bipath poset
and another poset, like $\R^{2}$; much like one embeds a totally ordered set to
obtain a fiber of the fibered barcode. Indeed, in~\cref{sec:fibered arccode} we
show two non-isomorphic persistence modules with equal fibered barcode, but that
can be distinguished by considering bipath modules via order-preserving maps from bipaths. This suggests the definition of the \deff{fibered arc code}
given in~\cref{sec:fibered arccode}, closely related to the fibered barcode.

Given the role they play in understanding persistence modules that always decompose into intervals, we aim to further the theoretical and computational understanding of bipath modules. Such a study has been started by Aoki, Escolar and Tada~\cite{aokiSummandinjectivityIntervalCovers2023,aokiBipathPersistence2024}. They define persistence diagrams in this context, and provide a matrix-based algorithm to decompose bipath persistence modules~\cite{aokiBipathPersistence2024}. Here, among other consequences detailed below, we show that a bipath module can be decomposed by directly decomposing a related (finite) zigzag module. This allows us to use the zigzag decomposition algorithm, in the case of the modules being the homology of a simplex-wise filtration, of Dey and Hou~\cite{deyFastComputationZigzag2022a}, which works by decomposing a related standard single-parameter persistence module. Thus, any improvement to the decomposition of standard persistence transfers to zigzag persistence and then to bipath persistence.

\subsection{Contributions} We define a poset map from an infinite zigzag to the bipath, which we call the \deff{covering map} and is inspired by the universal covering of the circle. See~\cref{fig:zigzag_cover} for an illustration.

\begin{figure}
  \centering
  \includegraphics{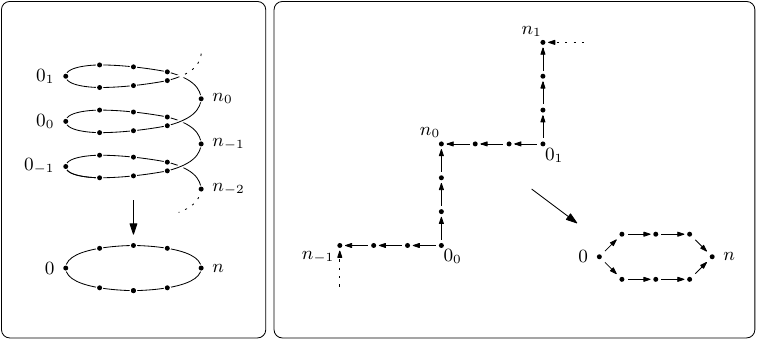}
  \caption{Illustration of the covering map defined in~\cref{sec:bipath_as_zigzag}}\label{fig:zigzag_cover}
\end{figure}

We translate elements of the theory of zigzag persistence to bipath persistence using this covering map. In~\cref{sec:bipath_as_zigzag}, we use the restriction functor induced by the covering map to obtain an infinite zigzag persistence module $R(M)$ from a bipath module $M$. Such a zigzag module $R(M)$ has an infinite but periodic barcode, and the barcodes of $R(M)$ and $M$ determine each other, in the sense of~\cref{thm:barcodes}.

In~\cref{sec:stability}, we establish algebraic stability for bipath persistence. This is a consequence of the algebraic stability of zigzag modules, which is, in turn, a consequence of the algebraic stability of block-decomposable modules~\cite{botnanAlgebraicStabilityZigzag2018,bjerkevikStabilityIntervalDecomposable2021}. Specifically, we define the interleaving and bottleneck distance between bipath modules $M$ and $N$ to be the interleaving and bottleneck distance between $R(M)$ and $R(N)$, as zigzag modules---much like one does for zigzag modules and block-decomposable modules~\cite{botnanAlgebraicStabilityZigzag2018,bjerkevikStabilityIntervalDecomposable2021}.

In~\cref{sec:computation}, we observe that to decompose the bipath module one does not need to consider the infinite zigzag module in its entirety, but just a finite slice of it, of size $O(K)$, where $K$ is the maximum length of a chain in the bipath poset.

Finally, in \cref{sec:fibered arccode} we define the fibered arc code for
persistence module over $\R^{2}$, and compare it with other invariants through
examples.

\subsection{Related work}

Aoki, Emerson and Tada initiated the study of bipath persistence and provided an algorithm to decompose bipath modules and a way to visualize such a decomposition~\cite{aokiBipathPersistence2024}. Their algorithm is based on column and row operations on matrices. In a recent seminar, Tada \cite{Tada2024APATG} has mentioned ongoing work on a stability theory for functions from a certain continuous version of bipath posets to topological spaces. A priori, this is unrelated to the algebraic stability theorem we state here, which is a direct consequence of the covering map of~\cref{sec:bipath_as_zigzag} that relates bipath to zigzag persistence.

Also recently, Asashiba and
Liu~\cite{asashibaIntervalMultiplicitiesPersistence2024} decompose a bipath module by decomposing two related zigzag modules, inspired by the notion of~\textit{essential cover} that they introduce. In our approach, we use a single zigzag, which, apart from simplifying the computation, allows us to translate the algebraic stability of zigzag modules to this setting, as mentioned.

In the context of circle-valued maps and level persistence, Burghelea and
Dey~\cite{burgheleaTopologicalPersistenceCircleValued2013} have studied
(pointwise finite-dimensional) persistence modules over finite subsets of the
circle $\bS^{1}$ equipped with a specific partial order. Such persistence modules decompose as a direct sum of interval modules and Jordan cells. Since their methods can be used to decompose zigzag modules, and as a consequence of our work, their methods could also be applied to decompose bipath modules. 

Sala and Schiffmann~\cite{salaFockSpaceRepresentation2021} have studied a similar finite persistence module over $\bS^{1}$ with the cyclic order. Hanson and Rock~\cite{hansonDecompositionPointwiseFinitedimensional2024} study
continuous pointwise finite-dimensional persistence modules over $\bS^{1}$
equipped with certain general partial orders (which include the cyclic order and the bipath order) and provide a decomposition theorem into intervals and Jordan cells. 
 Recently, Gao and Zhao~\cite{gaoIsometryTheoremContinuous2024} prove an isometry theorem for nilpotent persistence modules over $\bS^{1}$ with the cyclic order.

\subsection{Acknowledgments}
The authors would like to thank Prof. Hideto Asashiba for helpful
conversations on the naming of the arc code, and fruitful discussions on the covering technique of representations over the quiver of type $\tilde{A}$.

\section{Bipath as zigzag}\label{sec:bipath_as_zigzag}

In this section, we cover the bipath poset with an infinite zigzag, and describe its consequences.

We take a functorial perspective. A \deff{persistence module} over a poset $P$ is a functor $F\colon P\to\Vec$ from $P$, viewed as a category, to the category $\Vec$ of finite-dimensional vector spaces over a fixed field $\k$. We denote by $\Vec^P$ the category of persistence modules over $P$ where the morphisms are natural transformations. We denote the internal maps of a persistence module $F$ by $F_{p\to q}$ for all $p\leq q \in P$. A persistence module $F$ is \deff{indecomposable} if $F\iso A\oplus B$ then either $A = 0$ or $B = 0$.

Throughout this work, a \deff{multiset} $A$ is a pair $(S, \mu)$ where $S$ is a set and $\mu\colon S\to\N_{\ge 0}$ is a function that specifies the multiplicity of $s\in S$ in $A$. We consider actions on multisets in the following sense: a group $G$ has an action $\alpha$ on the multiset $A$ if $G$ has an action $\alpha\colon G\times S\to S$ on $S$ such that $\mu(\alpha(g, a)) = \mu(a)$. Such an action $\alpha$ is \deff{free} if $\alpha\colon G\times S\to S$ is free.

\subsection{Intervals in a bipath} Let us start by describing how bipath modules (that is, persistence modules over the bipath poset) decompose into interval modules, following Aoki, Emerson and Tada~\cite{aokiBipathPersistence2024}. Recall that an \deff{interval} $I$ of a poset $P$ is a connected non-empty subposet that is also convex, meaning that if for three $p\leq z \leq q$ we have that $p,q\in I$ then $z\in I$ too. The \deff{interval module supported on an interval $I\subset P$}, denoted by $\k I$, is the indecomposable persistence module (see~\cite[Proposition 2.2]{botnanAlgebraicStabilityZigzag2018}) given by
\begin{equation*}
  (\k I)_{p} =
  \begin{cases}
    \k, & \text{if $p\in I$,}\\
    0, & \text{otherwise,}
  \end{cases}
  \quad\text{with internal maps}\quad
  (\k I)_{p\to q} =
  \begin{cases}
    \Id, & \text{if $p,q\in I$,}\\
    0, & \text{otherwise.}
  \end{cases}
\end{equation*}

We work on a fixed bipath poset $B$. That is, we fix $n, m\in\N$ greater than $1$ and set $B \coloneqq \Set{0, \dots, n + m - 1}$ with the order given by two maximal chains $0 < 1 < \cdots < n$ and $0 < n + m - 1 < n + m - 2 < \cdots < n$:
\[
  \begin{tikzcd}[ampersand replacement=\&, column sep=small, row sep=small]
	\& 1 \& 2 \& \cdots \& n - 1 \\
	0 \&\&\&\&\& n. \\
	\& {n + m - 1} \& {n + m - 2} \& \cdots \& {n + 1}
	\arrow[from=1-2, to=1-3]
	\arrow[from=1-3, to=1-4]
	\arrow[from=1-4, to=1-5]
	\arrow[from=1-5, to=2-6]
	\arrow[from=2-1, to=1-2]
	\arrow[from=2-1, to=3-2]
	\arrow[from=3-2, to=3-3]
	\arrow[from=3-3, to=3-4]
	\arrow[from=3-4, to=3-5]
	\arrow[from=3-5, to=2-6]
\end{tikzcd}
\]

\newcommand{\lefti}{\vartriangleleft}
\newcommand{\righti}{\vartriangleright}
\newcommand{\barcode}{\mathcal{B}}
\newcommand{\full}{\mathrm{full}}
By the result of Aoki, Escolar and Tada~\cite[Theorem 1.3]{aokiSummandinjectivityIntervalCovers2023} mentioned in the
introduction, every persistence module over $B$ decomposes into interval modules. It is straightforward to verify that an interval of $B$ belongs to one of the following types (see~\cref{fig:intervals}).
\begin{enumerate}
  \item \deff{The full interval}, which is the poset $B$ itself,
  \item \deff{a left interval} for some
        $i\in \Set{0, 1, \dots, n-1}$ and $j\in\Set{0, n + m -1, \dots, n+1}$, such
        that
        \begin{equation*} [i,j]_{\lefti} \coloneqq \Set{z \in B \given 0 \leq z \leq i} \cup \Set{z\in B \given 0 \leq z \leq j},
        \end{equation*}
  \item \deff{a right interval} for some $i\in \Set{1, \dots, n-1, n}$ and $j\in\Set{n + m -1, \dots, n+1, n}$, such that
        \begin{equation*} [i,j]_{\righti} \coloneqq \Set{z \in B \given i \leq z \leq n} \cup \Set{z\in B \given j \leq z \leq n},
        \end{equation*}
  \item \deff{a top interval} for some $i, j\in \Set{1, \dots, n-1}$ such that
        \begin{equation*} [i,j]_{\top} \coloneqq \Set{z \in B \given i \leq z \leq j}, \text{
                and}
            \end{equation*}
  \item \deff{a bottom interval} for some $i, j\in \Set{n+ m - 1, \dots, n+1}$ such that
        \begin{equation*} [i,j]_{\bot} \coloneqq \Set{z \in B \given j \leq z \leq i}.
        \end{equation*}
\end{enumerate}
\begin{figure}[h]
  \centering
  \includegraphics{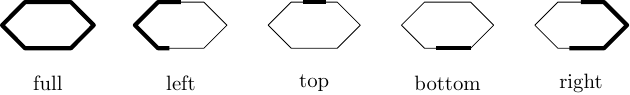}
  \caption{Types of intervals of the bipath poset}\label{fig:intervals}
\end{figure}

Since every bipath module is interval-decomposable, and by the Krull-Schmidt
theorem, there is a uniquely defined multiset of intervals of $B$ for every bipath module $M$, denoted by $\barcode(M)$, such that
\begin{equation}\label{eq:ind-decomp-of-M}
  M\iso \bigoplus_{I\in\barcode(M)} \k I.
\end{equation}
In what follows, instead of the \emph{bar}code, we would call $\barcode(M)$ the~\deff{arc code} of the bipath module $M$. See \cref{fig:bipath_arccode} for an illustration.

\begin{figure}[h]
  \centering
  \includegraphics{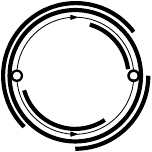}

  \caption{Illustration of an arc code}\label{fig:bipath_arccode}
\end{figure}

Considering the previously mentioned types of intervals, the arc code $\barcode(M)$ can be written as the union
\begin{equation*}
  \barcode(M) = \barcode_{\full}(M)\cup\barcode_{\lefti}(M)\cup\barcode_{\righti}(M)\cup\barcode_{\bot}(M)\cup\barcode_{\top}(M),
\end{equation*}
where $\barcode_{\full}(M)$ denotes the submultiset of $\barcode(M)$ consisting of full intervals, $\barcode_{\lefti}(M)$ denotes the submultiset consisting of left intervals, and analogously for the rest.

\subsection{Covering map}
The \deff{zigzag poset} $\ZZ$ is the subposet of $\Z\op\times\Z$ given by $\Set{(a, b) \in \Z\op\times\Z \given a = b \text{ or } a = b + 1}$. A zigzag module is a persistence module over the zigzag poset.

\begin{definition}\label{def:covering}
We define a poset map $\zeta\colon\ZZ\to B$ that we call the \deff{covering map} as follows, where we set $N \coloneqq n + m - 1$,
\begin{align*}
    &\zeta(kN+1, kN) = 0, & \\
    &\zeta(kN + i, kN + i) = \zeta(kN + i + 1, kN + i) = i, &\hspace{-1ex}\text{for } i\in \Set{1, \dots, n - 1},\\
    &\zeta(kN + n, kN + n) = n, \\
    &\zeta(kN - i + 1, kN - i) = \zeta(kN - i + 1, kN - i + 1) = n + m - i, &\hspace{-1ex}\text{for } i\in \Set{1, \dots, m - 1},
\end{align*}
for every $k\in \Z$. See~\cref{fig:covering} for an illustration.
\end{definition}

\begin{figure}[h]
  \centering
  \includegraphics{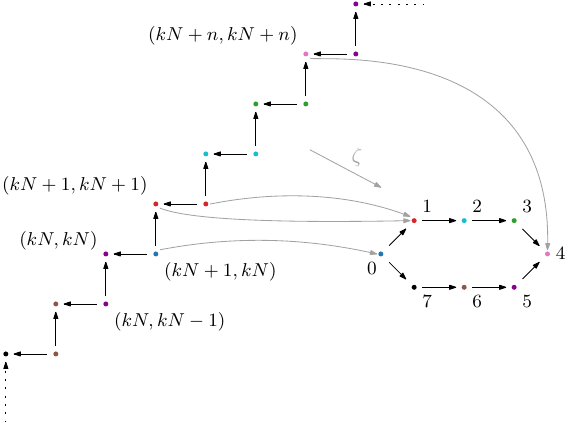}
  \caption{Illustration of the covering map $\zeta\colon\ZZ\to B$. In this case,
  $n = m = 4$}\label{fig:covering}
\end{figure}

The covering map $\zeta\colon\ZZ\to B$ induces a functor
$R\colon \Vec^{B}\to\Vec^{\ZZ}$, called the \deff{restriction functor},
by precomposition $M\mapsto M\circ\zeta$. It is easy to check, since our persistence modules are pointwise finite-dimensional, the restriction functor $R$ preserves direct sums.

Shortly, we describe the arc code of the bipath module $M$ in terms of the barcode of $R(M)$. Indeed, persistence modules over $\ZZ$ (when they are pointwise finite-dimensional, like here) have a barcode, that is, for every module $Z\in\Vec^{\ZZ}$ there is a unique multiset $\barcode(Z)$ of intervals such that $Z\iso \bigoplus_{I\in \barcode(Z)} \k I$, by a result of Botnan~\cite{botnanIntervalDecompositionInfinite2017}.

Following Bjerkevik~\cite{bjerkevikStabilityIntervalDecomposable2021}, we adopt
the following notation to refer to the intervals of $\ZZ$, where
$a, b\in\Z\cup\Set{\pm\infty}$:
\begin{align*}
  &[a,b]_{\ZZ} \coloneqq \Set{(c,d)\in\ZZ \given c\leq b, d\geq a},&[a, b)_{\ZZ} &\coloneqq \Set{(c,d)\in\ZZ \given a\leq d < b},\\
  &(a, b]_{\ZZ} \coloneqq \Set{(c, d)\in\ZZ \given a < c \leq b},&(a, b)_{\ZZ} &\coloneqq \Set{(c, d)\in\ZZ\given c > a, d < b}.
\end{align*}

We describe the image under the restriction functor
$R\colon \Vec^{B}\to\Vec^{\ZZ}$ of the different types of intervals of a bipath
$B$ in the following lemma, illustrated in~\cref{fig:zz_decomp}.

\begin{lemma}\label{lem:bip-zz-int-corr}
For each type of interval module over the bipath poset $B$, the following results hold.
\begin{itemize}
    \item For the full interval $B$, $R(\k B) = \k \ZZ$.
    \item For a left interval $[i,j]_{\lefti}$:
\begin{itemize}
    \item if $j=0$, then $R(\k [i, 0]_{\lefti}) = \bigoplus_{k\in \Z} \k(kN, kN + i + 1)_{\ZZ}$.
    \item if $j\neq 0$, then $R(\k [i, j]_{\lefti}) = \bigoplus_{k\in \Z} \k(kN - n - m + j, kN + i + 1)_{\ZZ}$.
\end{itemize} 
    \item For a right interval $[i, j]_{\righti}$, $R(\k [i,j]_{\righti}) = \bigoplus_{k\in \Z} \k [kN + i, kN + j]_{\ZZ}$.
    \item For a top interval $[i, j]_{\top}$, $R(\k [i,j]_{\top}) = \bigoplus_{k\in \Z} \k [kN + i, kN + j + 1)_{\ZZ}$.
    \item For a bottom interval $[i, j]_{\bot}$, $R(\k [i,j]_{\bot}) = \bigoplus_{k\in \Z} \k (kN + i - 1, kN + j]_{\ZZ}$.
  \end{itemize}
Recalling that $N\coloneqq n + m - 1$.
\end{lemma}
\begin{figure}
  \centering
  \includegraphics{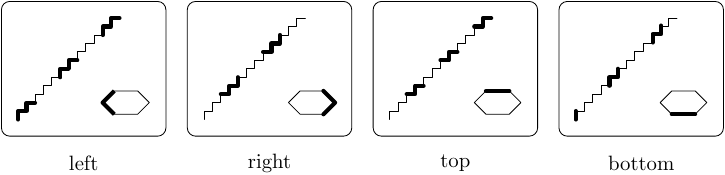}
  \caption{Illustration of interval modules over the bipath poset, and their image under the restriction functor $R$, as in~\cref{lem:bip-zz-int-corr}}~\label{fig:zz_decomp}
\end{figure}

Note that each $R(\k I)$ for an interval $I$ of $B$, except for the full interval, is an infinitely \textit{periodic} direct sum of interval modules over $\ZZ$ with disjoint support. Due to this periodicity and since $R$ preserves direct sums, for a bipath module $M$, we have a free action of $\Z$ on the barcode $\barcode(R(M))$, given by adding $zN$ to each endpoint of an interval $J$ for $z\in \Z$ and $J\in \barcode(R(M))$. We denote by $z\uparrow J\in \barcode(R(M))$ the result of $z$ acting on $J$. In conclusion, we obtain the following theorem.

\begin{theorem}\label{thm:barcodes}
Let $M$ be a bipath module and $\barcode(M)\coloneqq (S, \mu)$ the arc code of $M$. Then the barcode $\barcode(R(M))\coloneqq (S', \mu')$ of the zigzag module $R(M)$ is a complete invariant up to isomorphism.

Moreover, there is a bijection $\phi\colon S \iso S'/\Z$ such that $\mu'(\phi(I)) = \mu(I)$ holds for each $I\in S$. Here $S'/\Z$ denotes the orbit space of $S'$.
\end{theorem}

\section{Algebraic stability}\label{sec:stability}
As we have seen in~\cref{thm:barcodes}, the covering map $\zeta\colon\ZZ\to B$ allows us to translate the arc code of the bipath module to the barcode of the infinite zigzag module, and vice versa. Furthermore, and as we do shortly, the covering map can be used to study the algebraic stability of bipath modules, via the algebraic stability of zigzag modules, which is well-established following Botnan and Lesnick~\cite{botnanAlgebraicStabilityZigzag2018} and Bjerkevik~\cite{bjerkevikStabilityIntervalDecomposable2021}. These stability results depend on the concept of block-decomposable modules, which we review first after recalling the definitions of interleavings and bottleneck distance.

\subsection{Interleavings and matchings} We review interleavings over $\RR$ and the bottleneck distance. Recall that for an $\epsilon\in [0, \infty]$, we define the \deff{shift functor} $(\cdot)(\epsilon)\colon \Vec^{\RR} \to \Vec^{\RR}$ by sending $M$ to the persistence module given by $M(\epsilon)_{(x, y)} \coloneqq M_{(x-\epsilon, y+\epsilon)}$ and $M(\epsilon)_{(x, y)\to (w, z)} \coloneqq M_{(x-\epsilon, y+\epsilon)\to (w-\epsilon, z+\epsilon)}$. On morphisms $f\colon M\to N$, we set $f(\epsilon)_{(x, y)}\coloneqq f_{(x-\epsilon, y+\epsilon)}$. For every module $M$, we define a morphism $\phi^M_\epsilon\colon M\to M(\epsilon)$ by the internal maps $M_{(x, y)\to (x-\epsilon, y+\epsilon)}$. We say that two persistence modules $M, N\in\Vec^{\RR}$ are \deff{$\epsilon$-interleaved} if there exist morphisms $f\colon M\to N(\epsilon)$ and $g\colon N\to M(\epsilon)$ such that $g(\epsilon)\circ f = \phi^M_{2\epsilon}$ and $f(\epsilon)\circ g = \phi^N_{2\epsilon}$.

As in~\cite{bjerkevikStabilityIntervalDecomposable2021}, we describe matchings in terms of interleavings between interval modules. Let $A$ and $B$ be two multisets of intervals of $\RR$. An \deff{$\epsilon$-matching} between $A$ and $B$ is given by a bijection $\sigma\colon A'\to B'$ for some subsets $A'\subset A$ and $B'\subset B$ such that
\begin{itemize}
  \item every interval $I$ in $A\setminus A'$ and $B\setminus B'$ satisfies that
        $\k I$ is $\epsilon$-interleaved with the zero module, and
  \item for every $I\in A'$, the modules $\k I$ and $\k\sigma(I)$ are
        $\epsilon$-interleaved.
\end{itemize}
Then, the \deff{interleaving distance} $d_{I}(M, N)$ between two persistence modules $M, N\colon$ $\RR\to\Vec$ and the \deff{bottleneck distance} $d_{B}(C, D)$ between two multisets of intervals $C$ and $D$ are given by
\begin{align*}
  d_{I}(M, N) &\coloneqq \inf\Set{\epsilon \given \text{$M$ and $N$ are $\epsilon$-interleaved}},\\
  d_{B}(C, D) &\coloneqq \inf\Set{\epsilon \given \text{there is an $\epsilon$-matching between $C$ and $D$}}.
\end{align*}
In the sequel, for two persistence modules $M$ and $N$, $d_{B}(M, N)$ denotes the bottleneck distance between the barcodes $\barcode(M)$ and $\barcode(N)$. Namely, $d_{B}(M, N)\coloneqq d_{B}(\barcode(M), \barcode(N))$.

\subsection{Block-decomposable modules} Let $\cU$ be the subposet of $\RR$ given by $\Set{(a, b)\in \RR \given a\leq b}$. As in~\cite{botnanAlgebraicStabilityZigzag2018,bjerkevikStabilityIntervalDecomposable2021}, we are interested in the following four types of intervals of $\cU$.

\begin{definition}
The following intervals of $\cU$ are called \deff{blocks}, where $a, b\in \bbR\cup\Set{\pm\infty}$ and $a \le b$,
\begin{align*}
&[a,b]_{\BL} \coloneqq \Set{(c, d)\in \cU \given c\le b, d\ge a}, &[a, b)_{\BL} &\coloneqq \Set{(c, d)\in \cU \given a\le d<b},\\
&(a, b]_{\BL} \coloneqq \Set{(c, d)\in \cU \given a<c\le b}, &(a, b)_{\BL} &\coloneqq \Set{(c, d)\in \cU \given c>a, d<b}.
\end{align*}

A \deff{block module} is an interval module $\k I$, where $I$ is a block as defined above. A persistence module over $\cU$ is \deff{block-decomposable} if it decomposes into a direct sum of block modules.
\end{definition}

The following establishes the algebraic stability of block-decomposable modules and is due to Bjerkevik~\cite[Theorem 4.18]{bjerkevikStabilityIntervalDecomposable2021}.

\begin{theorem}[Block isometry theorem]\label{thm:block_isometry}
  For two block-decomposable modules $M$ and $N$, $d_{I}(M, N) = d_{B}(M, N)$.
\end{theorem}

We have the following on the structure of matchings between barcodes of
block-decomposable modules, see~\cite[Lemma
3.1]{botnanAlgebraicStabilityZigzag2018}.

\begin{lemma}\label{lem:structure_matchings}
  Let $M$ and $N$ be block-decomposable modules and consider an
  $\epsilon$-matching $\sigma\colon C\to D$ for subsets $C\subset\barcode(M)$
  and $D\subset\barcode(N)$. Then the blocks $I$ and $\sigma(I)$ are of the same type for
  every $I\in C$.
\end{lemma}

\subsection{Extension of zigzag to block-decomposable modules} Consider the inclusion $\io\colon\ZZ\hookrightarrow\RR$. By composing first the left Kan extension along $\io$ and then the restriction functor of the inclusion $\cU\to\RR$, one obtains the \deff{block extension functor} $E\colon\Vec^{\ZZ}\to\Vec^{\cU}$; we refer to~\cite{botnanAlgebraicStabilityZigzag2018,bjerkevikStabilityIntervalDecomposable2021} for more context on this functor and the definitions involved. The block extension functor preserves direct sums, and sends barcodes of zigzag modules to barcodes of block-decomposable modules, by which we mean the following. See~\cref{fig:blocks_zigzag} for an illustration.
\begin{lemma}\label{lem:corr_zigzag_blocks}
  Let $M$ be a zigzag module. Then $E(M)$ is block-decomposable, and
  \begin{align*}
    &E(\k [a,b]_{\ZZ}) = \k [a, b]_{\BL}, &E(\k[a, b)_{\ZZ}) &= \k[a, b)_{\ZZ},\\
    &E(\k (a, b]_{\ZZ}) = \k (a, b]_{\BL}, &E(\k(a, b)_{\ZZ}) &= \k (a, b)_{\BL}.
  \end{align*}
  Thus, $\barcode(E(M)) = E(\barcode(M))$.
\end{lemma}
\begin{figure}
  \centering
  \includegraphics{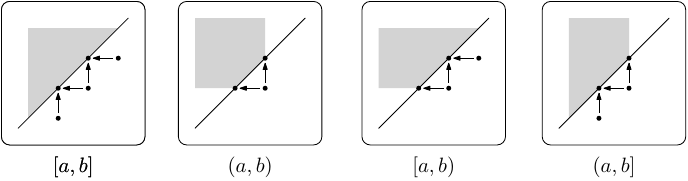}
  \caption{Illustration of the correspondence between the interval of $\ZZ$ and the block of $\cU$, as in~\cref{lem:corr_zigzag_blocks}}~\label{fig:blocks_zigzag}
\end{figure}

As in~\cite{botnanAlgebraicStabilityZigzag2018,bjerkevikStabilityIntervalDecomposable2021}, this lemma motivates the definition of interleaving and bottleneck distance of two zigzag modules $M$ and $N$ by setting
\begin{equation*}
  d_{I}(M, N) \coloneqq d_{I}(E(M), E(N)) \text{ and } d_{B}(M, N) \coloneqq d_{B}(E(M), E(N)).
\end{equation*}
As such,~\cref{thm:block_isometry} translates directly to zigzag modules.

\subsection{Algebraic stability of bipath persistence} Recall the restriction functor
$R\colon \Vec^{B}\to\Vec^{\ZZ}$ induced by the covering map defined in~\cref{def:covering}. Also recall that by~\cref{thm:barcodes}, $R$ preserves arc codes up to the action of $\Z$. In turn, and in parallel to the definition of interleaving and bottleneck distances in zigzag modules, for two bipath modules $M$ and $N$, we define
\begin{equation*}
  d_{I}(M, N) \coloneqq d_{I}(R(M), R(N)) \text{ and } d_{B}(M, N) \coloneqq d_{B}(R(M), R(N)).
\end{equation*}
Combining this with the block isometry theorem (\cref{thm:block_isometry}), we immediately obtain the isometry theorem for bipath persistence.
\begin{theorem}[Bipath isometry theorem]
Let $M$ and $N$ be bipath modules. Then $d_{I}(M, N) = d_{B}(M, N)$.
\end{theorem}

Recall that the $\epsilon$-matching between $\barcode(R(M))$ and $\barcode(R(N))$ is indeed the matching between infinite periodic barcodes, but can be made equivariant in the following sense:

\begin{lemma}
Let $M$ and $N$ be bipath modules. If there exists an $\epsilon$-matching between $\barcode(R(M))$ and $\barcode(R(N))$, then there exists an $\epsilon$-matching given by a bijection $\sigma\colon C\to D$, for some $C\subset \barcode(R(M))$ and $D\subset\barcode(R(N))$ that is equivariant with respect to the action of $\Z$, meaning that
  \begin{equation*}
    \sigma(z\uparrow I) = z\uparrow \sigma(I),
  \end{equation*}
  for every $I\in C$.
\end{lemma}
\begin{proof}
Take an $\epsilon$-matching between $\barcode(R(M))$ and $\barcode(R(N))$ given
by the bijection $\tau\colon E\to F$. Suppose that there exist some $z\in \Z$
and $I\in E$ such that $\tau(z\uparrow I) \neq z\uparrow\tau(I)$. For such the
interval $I$, we can construct a new $\epsilon$-matching given by the bijection
$\sigma\colon E\cup\Set{z'\uparrow I\given z'\in\Z}\to F\cup\Set{z'\uparrow \tau(I)\given z'\in\Z}$,
where $\sigma(z'\uparrow I)\coloneqq z'\uparrow \tau(I)$ for every $z'\in\Z$ and
$\sigma(J)\coloneqq \tau(J)$ for any $J\in E$ that is not in the orbit
$\Set{z'\uparrow I\given z'\in\Z}$ of $I$. We note the reader that
$\sigma(I) = \sigma(0\uparrow I) = 0\uparrow \tau(I) = \tau(I)$ by the
definition, it then follows that $\sigma(z'\uparrow I) = z'\uparrow \sigma(I)$
for every $z'\in \Z$. Finally, since there is a finite number of orbits in $\barcode(R(M))$, we can repeat the procedure just described to obtain the desired $\epsilon$-matching.
\end{proof}

Such equivariance suggests the following definition of $\epsilon$-matching between arc codes of bipath modules.
\begin{definition}
Let $M$ and $N$ be bipath modules. An $\epsilon$-matching between the arc codes $\barcode(M)$ and $\barcode(N)$ is given by a bijection $\sigma\colon C\to D$ for two subsets $C\subset\barcode(M)$ and $D\subset\barcode(N)$ such that
  \begin{itemize}
    \item every interval $I$ in $\barcode(M)\setminus C$ and
          $\barcode(N)\setminus D$ satisfies that $E\circ R(\k I)$ is
          $\epsilon$-interleaved with the zero module, and
    \item for every $I\in C$, the modules $E\circ R(\k I)$ and $E\circ R(\k\sigma(I))$
          are $\epsilon$-interleaved.
  \end{itemize}
\end{definition}
It is now clear that an equivalent definition of the bottleneck distance between bipath modules $M$ and $N$ can be given by
\begin{equation*}
  d_{B}(M, N) \coloneqq \inf\Set{\epsilon \given \text{there is an $\epsilon$-matching between $\barcode(M)$ and $\barcode(N)$}},
\end{equation*}
and such the $\epsilon$-matchings need to match intervals of the same type in the bipath poset, as in the situation for block-decomposable modules of~\cref{lem:structure_matchings}.

\section{Computation}\label{sec:computation}
We have related bipath persistence with an infinite periodic zigzag. We now show how to actually compute it, by taking a finite slice. This will allow us to decompose any bipath persistence module by decomposing a single finite zigzag.

Let $\slice{\ZZ}\subset\ZZ$ be the subposet of $\ZZ$ given by $\slice{\ZZ}\coloneqq (-m+1, n+m)_{\ZZ}$, which contains $(2n+4m-3)$ elements. We write $\slice{\zeta}\colon\slice{\ZZ} \to B$ for the restriction of the covering map $\zeta$ to $\slice{\ZZ}$. As before, we write $\slice{R}$ as the restriction functor induced by $\slice{\ze}$. In the following, we abuse the notation $(a,b)_{\ZZ}$ (resp. $[a,b]_{\ZZ}$, $[a,b)_{\ZZ}$, $(a,b]_{\ZZ}$) also for the interval $(a,b)_{\slice{\ZZ}}$ in $\slice{\ZZ}$ if there seems to be no ambiguity.

Since an arc code is a multiset of intervals, we write $\mu_{M}(I)$ for the multiplicity of the interval $I$ in the arc code $\barcode(M)$.

\begin{table}
\centering
\begin{tabular}{cc}
\hline
interval over $B$ & interval over $\slice{\ZZ}$ \\ \hline
\( B \) & \( \slice{\ZZ} \) \\
\( [i, 0]_{\lefti} \) & \( (0, i + 1)_{\ZZ} \) \\
\( [i, j]_{\lefti}\ (j\neq 0) \) & \( (- n - m + j, i + 1)_{\ZZ} \) \\
\( [i, j]_{\righti} \) & \( [i, j]_{\ZZ} \) \\
\( [i,j]_{\top} \) & \( [i, j + 1)_{\ZZ} \) \\
\( [i,j]_{\bot} \) & \( (i - 1, j]_{\ZZ} \) \\ \hline
\end{tabular}
\caption{Correspondence between the intervals in the equalities of multiplicities}\label{tab:equalities}
\end{table}

\begin{theorem}\label{thm:main-theorem}
The multiplicity of each interval in the arc code can be obtained from the multiplicity of the corresponding interval in the barcode of the finite zigzag $\slice{\ZZ}$, as in~\cref{tab:equalities}. Namely, the following equalities hold for any bipath module $M$.

\begin{itemize}
    \item For the full interval $B$, $\mu_M(B) = \mu_{\slice{R}(M)}(\slice{\ZZ})$.
    \item For a left interval $[i,j]_{\lefti}$:
\begin{itemize}
    \item if $j=0$, then $\mu_M(\k [i, 0]_{\lefti}) = \mu_{\slice{R}(M)}\left((0, i + 1)_{\ZZ}\right)$.
    \item if $j\neq 0$, then $\mu_M([i, j]_{\lefti}) = \mu_{\slice{R}(M)}\left((- n - m + j, i + 1)_{\ZZ}\right)$.
\end{itemize}  
    \item For a right interval $[i, j]_{\righti}$, $\mu_M([i,j]_{\righti}) = \mu_{\slice{R}(M)}\left([i, j]_{\ZZ}\right)$.
    \item For a top interval $[i, j]_{\top}$, $\mu_M([i,j]_{\top}) = \mu_{\slice{R}(M)}\left([i, j + 1)_{\ZZ}\right)$.
    \item For a bottom interval $[i, j]_{\bot}$, $\mu_M([i,j]_{\bot}) = \mu_{\slice{R}(M)}\left((i - 1, j]_{\ZZ}\right)$.
  \end{itemize}
\end{theorem}

\begin{proof}
Since every bipath module is interval-decomposable, we write the indecomposable decomposition of the bipath module $M$ as
\begin{align*}
    M \iso & (\k B)^{\mu_M(B)}\ds \DS\k [i, j]_{\lefti}^{\mu_M([i, j]_{\lefti})} \ds \DS\k [i, j]_{\righti}^{\mu_M([i, j]_{\righti})}\\
    &\quad\quad \ds \DS\k [i, j]_{\top}^{\mu_M([i, j]_{\top})}\ds \DS\k [i, j]_{\bot}^{\mu_M([i, j]_{\bot})}.
\end{align*}
Applying the restriction functor $\slice{R}$ yields the decomposition of $\slice{R}(M)$:
\begin{align}\label{eq:res-of-M-decomp}
    \slice{R}(M) \iso & \slice{R}(\k B)^{\mu_{\slice{R}(M)}(\ZZ)}\ds \DS \slice{R}(\k [i, j]_{\lefti})^{\mu_M([i, j]_{\lefti})} \ds \DS \slice{R}(\k [i, j]_{\righti})^{\mu_M([i, j]_{\righti})}\nonumber\\
    &\quad\quad \ds \DS \slice{R}(\k [i, j]_{\top})^{\mu_M([i, j]_{\top})}\ds \DS \slice{R}(\k [i, j]_{\bot})^{\mu_M([i, j]_{\bot})}.
\end{align}
We focus on the restriction of interval modules in \eqref{eq:res-of-M-decomp}. It is not difficult to see the following from Lemma~\ref{lem:bip-zz-int-corr}.
\begin{itemize}
\item For the full interval $B$, $\slice{R}(\k B) = \k \slice{\ZZ}$.
\item For the left interval $[i, j]_{\lefti}$, we have
\[
\slice{R}(\k [i, j]_{\lefti}) = \begin{cases}
\k (0, i+1)_{\ZZ}, & j = 0,\\
\k (-n-m+j, i+1)_{\ZZ}\ds \k (j-1, n+m)_{\ZZ}, & j \neq 0.
\end{cases}
\]
\item For the right interval $[i, j]_{\righti}$, we have
\[
\slice{R}(\k [i, j]_{\righti}) = \begin{cases}
\k [i, n]_{\ZZ}, & j = n,\\
\k [i, j]_{\ZZ}\ds \k (-m+1, -n-m+j+1]_{\ZZ}, & j \neq n.
\end{cases}
\]
\item For the top interval $[i,j]_{\top}$, $\slice{R}(\k [i,j]_{\top}) = \k [i,j+1)_{\ZZ}$.
\item For the bottom interval $[i,j]_{\bot}$, $\slice{R}(\k [i,j]_{\bot}) = \k (i-1,j]_{\ZZ}\ds \k (-n-m+i, -n-m+j+1]_{\ZZ}$.
\end{itemize}

From the above list, it turns out that every zigzag module that appeared in the right-hand side of multiplicity equalities is the indecomposable direct summand of the corresponding $\slice{R}(\k I)^{\mu_M(\k I)}$, and is not the direct summand of $\DS_{J\neq I} \slice{R}(\k J)^{\mu_M(\k J)}$ in \eqref{eq:res-of-M-decomp}. Therefore, the assertions follow.
\end{proof}

\begin{remark}
We emphasize here that to compute the multiplicity of an interval $I$ it is crucial to find the corresponding indecomposable module of $\slice{R}(M)$. For example, to compute $\mu_M([0, j]_{\lefti})$ ($j\neq 0$), the interval $(j-1, n+m)_{\ZZ}$ is not the desired corresponding interval because it also corresponds to the direct summand $\slice{R}(\k [1, j]_{\lefti})$.
\end{remark}

\section{Fibered bar and arc codes}\label{sec:fibered arccode}
Lesnick and Wright \cite{lesnickInteractiveVisualization2D2015} introduced the
\deff{fibered barcode}, an invariant of persistence modules over $\R^{2}$ given
by the collection of barcodes of all 1-D affine slices. In detail, denoting by $L$
the collection of all (unparametrized) lines in $\R^{2}$ with non-negative slope
and for a fixed persistence module $M\colon\R^{2}\to\Vec$, each line $l\in L$
induces a single-parameter persistence module $M|_{l} = M\circ i_{L}$ along the
inclusion $i_{L}\colon l\hookrightarrow\R^{2}$. The \deff{fibered barcode} of $M$ is then the assignment
of each line $l\in L$ to the barcode of $M|_{l}$.

Closely following the construction of the fibered barcode, and due to bipath
modules having arc codes, given a bipath poset $B$ and a persistence module $M\colon\R^{2}\to\Vec$, an order-preserving
map $f\colon B\to\R^{2}$ induces a bipath module $M|_{f} = M\circ f$ and we can
define the \deff{fibered arc code}:

\begin{definition}\label{def:fibered arccode}
  The \deff{fibered arc code} of a persistence module
  $M\colon \mathbb{R}^{2} \to \text{Vec}$ is the assignment of each
  order-preserving map $f\colon B \to \mathbb{R}^{2}$, where $B$ is a bipath
  poset, to the arc code of $M|_{f}$.
\end{definition}

As already known, the fibered barcode is an incomplete invariant. The same
applies to the fibered arc code. Still, the fibered arc code gives more
information than the fibered barcode: in~\cref{exm:fibered arccode_vs_fibered
  barcode} below we show two non-isomorphic persistence modules with equal
fibered barcode, but different fibered arc code. In addition,
in~\cref{exm:fibered arccode_vs_interval rank invariant} we show two
non-isomorphic persistence modules with equal fibered arc code. To distinguish
the modules in this last example, one would need to use different invariants
like the~\emph{interval rank invariant}, proposed by Asashiba, Gauthier and
Liu~\cite{asashiba2024interval}, or other related invariants like the
\emph{generalized rank invariant}, proposed by Kim and
M\'emoli~\cite{kimGeneralizedPersistenceDiagrams2021}.

\begin{example}\label{exm:fibered arccode_vs_fibered barcode}
We take an example from Asashiba, Gauthier and Liu~\cite[Example 6.2]{asashiba2024interval}.
Let $G$ be the grid
\[
G\coloneqq\begin{tikzcd}[ampersand replacement=\&, column sep=1em, row sep=2em, nodes={scale=0.8}]
	{(2,1)} \& {(2,2)} \& {(2,3)} \& {(2,4)} \& {(2,5)} \\
	{(1,1)} \& {(1,2)} \& {(1,3)} \& {(1,4)} \& {(1,5)}
	\arrow[from=1-1, to=1-2]
	\arrow[from=1-2, to=1-3]
	\arrow[from=1-3, to=1-4]
	\arrow[from=1-4, to=1-5]
	\arrow[from=2-1, to=1-1]
	\arrow[from=2-1, to=2-2]
	\arrow[from=2-2, to=1-2]
	\arrow[from=2-2, to=2-3]
	\arrow[from=2-3, to=1-3]
	\arrow[from=2-3, to=2-4]
	\arrow[from=2-4, to=1-4]
	\arrow[from=2-4, to=2-5]
	\arrow[from=2-5, to=1-5]
\end{tikzcd}\]
and let $M_\lambda$, for a $\lambda\in\k$, be the following persistence module over  $G$:
\[
M_\lambda \coloneqq
\begin{tikzcd}[column sep=1.5em, row sep=1em, nodes={scale=1.2}]
\k & {\k^2} & {\k^2} & \k & 0 \\
0 & \k & {\k^2} & {\k^2} & \k
\arrow["{\binom{1}{0}}", from=1-1, to=1-2]
\arrow["\Id", from=1-2, to=1-3]
\arrow["{(\lambda,-1)}", from=1-3, to=1-4]
\arrow[from=1-4, to=1-5]
\arrow[from=2-1, to=1-1]
\arrow[from=2-1, to=2-2]
\arrow["{\binom{0}{1}}"', from=2-2, to=1-2]
\arrow["{\binom{0}{1}}"', from=2-2, to=2-3]
\arrow["\Id"', from=2-3, to=2-4]
\arrow["{(1,-1)}"', from=2-4, to=2-5]
\arrow["\Id"', from=2-3, to=1-3]
\arrow["{(\lambda,-1)}"', from=2-4, to=1-4]
\arrow[from=2-5, to=1-5]
\end{tikzcd}.
\]
We focus on the two non-isomorphic modules $M_1$ and $M_{-1}$. Let $B$ the
bipath poset with $n=3$ and $m= 1$ and consider the order-preserving map
$f\colon B\to G$ whose image is
\[
  \begin{tikzcd}[ampersand replacement=\&, column sep=1em, row sep=2em]
	{(2,3)} \& {(2,4)} \& {(2,5)} \\
	{(1,3)} \&\& {(1,5)}
	\arrow[from=1-1, to=1-2]
	\arrow[from=1-2, to=1-3]
	\arrow[from=2-1, to=1-1]
	\arrow[from=2-1, to=2-3]
	\arrow[from=2-3, to=1-3]
\end{tikzcd}.\]
Then the arc code $\barcode(M_1|_f)$ consists of two intervals of $B$ that we
depict by
\[
\begin{tikzcd}[ampersand replacement=\&, column sep=1em, row sep=2em]
	{(2,3)} \& {(2,4)} \\
	{(1,3)} \&\& {(1,5)}
	\arrow[from=1-1, to=1-2]
	\arrow[from=2-1, to=1-1]
	\arrow[from=2-1, to=2-3]
\end{tikzcd}
\text{and}
\begin{tikzcd}[ampersand replacement=\&]
	{(2,3)} \\
	{(1,3)}
	\arrow[from=2-1, to=1-1]
\end{tikzcd},
\]
while the arc code $\barcode(M_{-1}|_f)$ consists of two other intervals of $B$ that we depict by
\[
\begin{tikzcd}[ampersand replacement=\&]
	{(2,3)} \\
	{(1,3)} \&\& {(1,5)}
	\arrow[from=2-1, to=1-1]
	\arrow[from=2-1, to=2-3]
\end{tikzcd}
\text{and}
\begin{tikzcd}[ampersand replacement=\&, column sep=1em, row sep=2em]
	{(2,3)} \& {(2,4)} \\
	{(1,3)}
	\arrow[from=1-1, to=1-2]
	\arrow[from=2-1, to=1-1]
\end{tikzcd}.
\]
We conclude that the fibered arc codes of $M_1$ and $M_{-1}$ are different.

On the other hand, $M_1$ and $M_{-1}$ have the same fibered barcodes. It suffices to check three affine lines in red, as illustrated in \cref{fig:affine_lines}.
\begin{figure}[h]
  \centering
  \includegraphics[scale=0.6]{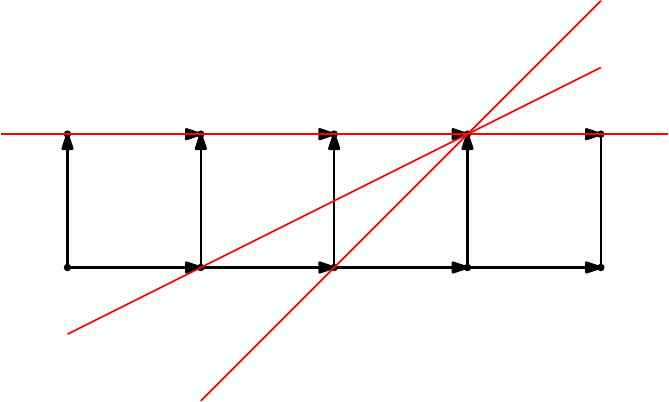}
  \caption{Illustration of the three affine lines used in~\cref{exm:fibered arccode_vs_fibered barcode}}\label{fig:affine_lines}
\end{figure}

For the red line $L_0$ with slope 0, $M_1|_{L_0}$ and $M_{-1}|_{L_0}$ have the same barcodes:
\[
\begin{tikzcd}
	{(2,1)} & {(2,2)} & {(2,3)} & {(2,4)}
	\arrow[from=1-1, to=1-2]
	\arrow[from=1-2, to=1-3]
	\arrow[from=1-3, to=1-4]
\end{tikzcd}
\text{and}
\begin{tikzcd}
	{(2,2)} & {(2,3)}
	\arrow[from=1-1, to=1-2]
\end{tikzcd}.
\]
The same happens for the red line $L_{1}$ with slope $1$, where the barcodes are
given by
\[
\begin{tikzcd}
	{(1,3)} & {(2,4)}
	\arrow[from=1-1, to=1-2]
\end{tikzcd}
\text{and}
\begin{tikzcd}
	{(1,3)}
\end{tikzcd},
\]
and for the red line $L_{1/2}$ with slope $1/2$, where the barcodes are
\[
\begin{tikzcd}
	{(1,2)} & {(2,4)}
	\arrow[from=1-1, to=1-2]
\end{tikzcd}.
\]
\end{example}

\begin{example}\label{exm:fibered arccode_vs_interval rank invariant}
Consider two bifiltrations on $3$ points illustrated in \cref{fig:Box_Examples}.
\begin{figure}[h]
  \centering
  \includegraphics[scale=0.6]{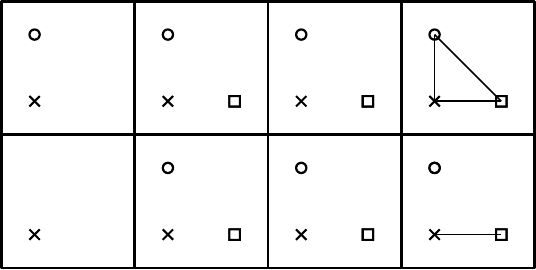}
  \quad\quad
  \includegraphics[scale=0.6]{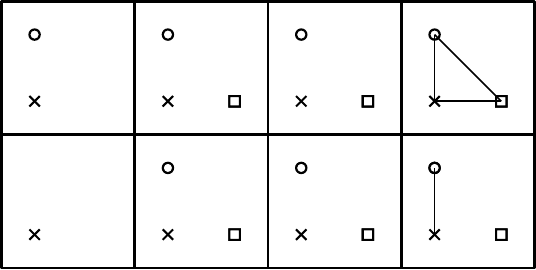}
  \caption{Two different filtrations}\label{fig:Box_Examples}
\end{figure}

We apply the 0-th reduced homology, with the coefficient field
$\k \coloneqq \Z_2$, on both. The persistent homology obtained from the left
(resp. right) filtration is denoted by $M$ (resp. $N$). It is not hard to verify
that $M$ and $N$ share the same fibered arc code. For instance, restricting to
the order-preserving map $f\colon B\to Q$, where $Q$ is the grid, depicted by
\begin{center}
  \includegraphics[scale=0.6]{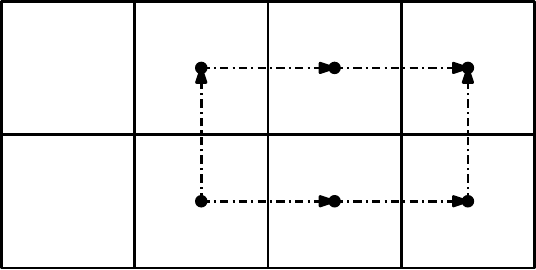}
\end{center}
yields that arc codes $\barcode(M|_{f})$ and $\barcode(N|_{f})$ share the same intervals as illustrated below:
\begin{center}
  \includegraphics[scale=0.6]{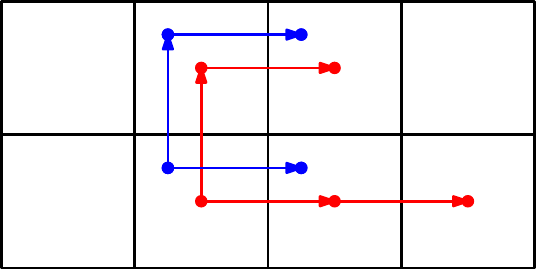}
\end{center}

On the other hand, we can distinguish these two filtrations by using the
\deff{interval rank invariant}~\cite{asashiba2024interval}. More concretely, if we
take the interval $J$ of the grid illustrated in \cref{fig:interval_in_CL4} and
use the source-sink compression system (see~\cite{asashiba2024interval}), then
the $J$-rank of $M$ is $1$ while that of $N$ is $0$.
\begin{figure}[h]
  \centering
  \includegraphics[scale=0.6]{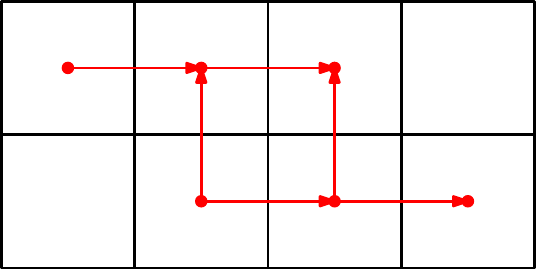}
  \caption{An interval of the grid}\label{fig:interval_in_CL4}
\end{figure}
\end{example}

\bibliographystyle{plainurl_fulljournal}
\bibliography{refs.bib}

\end{document}